\documentclass[12pt]{article}
\usepackage{amsmath,amsfonts,amsthm,amssymb, mathtools}
\usepackage{mathrsfs, graphicx,color,latexsym, tikz, calc,cite,enumerate,indentfirst}
\usetikzlibrary{shadows}
\usetikzlibrary{patterns,arrows,decorations.pathreplacing}
\usepackage{standalone}
\usepackage[colorlinks=true,
linkcolor=blue,citecolor=blue,
urlcolor=blue]{hyperref}
\usepackage[numbers,sort&compress]{natbib}
\usepackage[capitalise]{cleveref}
\newtheorem{theorem}{Theorem}[section]
\newtheorem{lemma}{Lemma}[section]
\newtheorem{problem}{Problem}[section]

\newtheorem{corollary}{Corollary}[section]

\newtheorem{claim}{Claim}
\theoremstyle{definition}

\newtheorem{case}{Case}
\newtheorem{subcase}{Subcase}[case]

\DeclareMathOperator{\Lip}{\mathsf{Lip}}

\DeclareMathOperator{\supp}{supp}

\allowdisplaybreaks

\title{\bf Graphs with positive Lin-Lu-Yau curvature without quadrilaterals}

\author{Huiqiu Lin,\ Zhe You\setcounter{footnote}{-1}\footnote{\emph{Email address:} huiqiulin@126.com(H. Lin),\ y30231280@mail.ecust.edu.cn (Z. You)}\\[2mm]
	\small School of Mathematics, East China University of Science and Technology, \\
	\small  Shanghai 200237, P. R. China
}

\date{}

\begin{document}
	\maketitle
	\begin{abstract}
    The definition of Ricci curvature on graphs was given in Lin-Lu-Yau, Tohoku Math., 2011, which is a variation of Ollivier, J. Funct. Math., 2009. 
    Recently, a powerful limit-free formulation of Lin-Lu-Yau curvature using the graph Laplacian has been given in M\"{u}nch-Wojciechowski, Adv. Math., 2019. 
    Let $F_k$ be the friendship graph obtained from $k$ triangles by sharing a common vertex and $T$ be the graph obtained from a triangle and $K_{1,3}$ by adding a matching between every leaf of $K_{1,3}$ and a vertex of the triangle.
    In this paper, we classify all the simple connected $C_4$-free graphs with positive Lin-Lu-Yau curvature for minimum degree at least 2: the cycles $C_3,C_5$, the friendship graphs $F_2,F_3$, the line graph of Peterson graph, and $T$.

		\par\vspace{2mm}
		
		\par\vspace{1mm}
		
		\noindent{\bfseries }
	\end{abstract}

\section{Introduction}\label{section::1}
Let $G=(V,E)$ be an undirected simple connected graph and $V,E$ respectively denote the set of vertices and edges of $G$.
Let $d_v$ denote the number of edges which contain $v$.
It is well known that the concept of curvature plays an important role in  Riemannian geometry. 
Over the past two decades, many scholars have studied a notion of curvature in graph theory, which is called the {\it combinatorial  curvature}. 
For each $v\in V(G)$, the combinatorial curvature at $v$, denoted by $\phi(v)$, is defined as
   
\begin{align*}
    \phi(v)=1-\frac{d_v}{2}+\sum\limits_{\sigma\in F(v)}\frac{1}{|\sigma|}.
\end{align*}

In the definition, $|\sigma|$ denotes the number of edges bounding $\sigma$, and $F(v)$ denotes the set of faces which $v$ is incident to. 
A graph $G$ is said to {\it have positive combinatorial curvature (everywhere)} if $\phi(v)>0$ for all $v\in V(G)$. 
Higuchi~\cite{Higuchi} conjectured that if a simple connected graph $G$ with positive combinatorial curvature (everywhere) can be embedded into a $2$-sphere and has minimum degree $\delta(G)\geq 3$, then $G$ is finite. 
Higuchi's conjecture was verified by Sun and Yu~\cite{Sun} for cubic planar graphs, and finally DeVos and Mohar~\cite{Devos} gave the answer as the following theorem.

\begin{theorem}[\cite{Devos}]\label{thm::1.1}
    Let $G$ be a simple connected graph which is $2$-cell embedded into a surface $S$ so that every vertex and face has degree at least $3$. 
    If $G$ has positive combinatorial curvature (everywhere), then  $G$ is finite and $S$ is homeomorphic to either {\color{blue}a} $2$-sphere or the projective plane. 
    Furthermore, if $G$ is not a prism, antiprism, or the projective planar analogue of one of these, then $|V (G)|\leq 3444$.
\end{theorem}

According to \cref{thm::1.1}, a natural question is what are the minimum constants for $|V(G)|$ above for $G$ such that it can be embedded into a $2$-sphere and the  projective plane, respectively. 
This problem was studied in~\cite{Chen,Reti,Zhang,Oh,Nicholson}. In particular, Nicholson and Sneddon~\cite{Nicholson} constructed some graphs of positive combinatorial curvature (everywhere) with $208$ vertices embedded into a $2$-sphere to obtain a lower bound of $|V(G)|$, and Ghidelli~\cite{Ghidelli} completely resolved this problem by getting the upper bound 208. 
Planar graphs with non-negative combinatorial curvature have also been investigated by Hua and Su~\cite{Hua1,Hua2,Hua3}.

 As we all know, the Ricci curvature plays a preeminent role in geometric analysis, which has been introduced to general metric measure spaces by Bakry-\'Emery~\cite{Bakry}, Ollivier~\cite{Ollivier}, Lott-Villani~\cite{LV} and Sturm~\cite{St1,St2}, etc. 
 This concept was extended to discrete spaces by Sammer~\cite{Sammer} while it is
 Chung and Yau~\cite{Chung} who first defined the Ricci curvature of a graph. 
 Recently, Lin, Lu and Yau~\cite{LLY} defined a new Ricci curvature on graphs by modifying Ollivier's Ricci curvature, which is called {\it Lin-Lu-Yau (Ricci) curvature} in the sequel. 
 The present paper is devoted to the investigation of  Lin-Lu-Yau curvature. We abbreviate it as LLY curvature for convenience.

 A graph is said to be {\it positively LLY-curved} if its LLY  curvature is positive everywhere. 
 Motivated by Higuchi's conjecture, Lu and Wang~\cite{planar LLY} studied the order of positively LLY-curved planar graphs.
 Besides, $\{C_3,C_5\}$-free connected positively LLY-curved graph was considered in~\cite{Lu C3}, which provides a tight bound on $|V(G)|$.
Lu and Wang~\cite{planar LLY} considered the bound on $|V(G)|$ of positively LLY-curved planar graphs with $\delta(G)\geq3$. Last but not least, Brooks, Osaye, Schenfisch, Wang and Yu~\cite{outerplanar} studied positively LLY-curved outerplanar graphs and obtained a sharp upper bound of maximum degree $\Delta(G)$. 

 \begin{theorem}[\cite{outerplanar}]
    Let $G$ be a simple positively LLY-curved outerplanar graph with $\delta(G)\geq 2$. 
    Then $\Delta(G)\leq 9$ and the upper bound is sharp.
\end{theorem}

 Moreover, they considered the bound on $|V(G)|$ when the connected positively LLY-curved outerplanar graph is maximal. Recently, Liu, Lu and Wang~\cite{LLW24} removed the condition “maximal”.
 
\begin{theorem}[\cite{LLW24}]
    Let $G$ be a simple connected positively LLY-curved  outerplanar graph with $\delta(G)\geq 2$. 
    Then $|V(G)|\leq 10$ and the upper bound is sharp.
\end{theorem}

The condition $\delta(G)\geq 2$ here is the origin of our condition. 
In this paper, we study positively LLY-curved $C_4$-free graphs with $\delta(G)\geq2$. 
Our motivation comes from~\cite{Lu C3} and~\cite{outerplanar}. 
In~\cref{section::3}, we will present a sharp upper bound of the maximum degree of a $C_4$-free graph with positive LLY curvature.

\begin{theorem}\label{thm::1.4}
    Let $G$ be a simple positively LLY-curved $C_4$-free graph with $\delta(G)\geq 2$. 
    Then $\Delta(G)\leq 6$ and the upper bound is sharp.
\end{theorem}

 The upper bound of maximum degree $6$ can be achieved by graph $F_3$ , which is also an outerplanar graph. 
 Here $F_k$ is the friendship graph obtained from $k$ triangles by sharing a common vertex. 
 In fact, we can further classify all connected graphs that satisfy the condition in \cref{thm::1.4} and have the following theorem.
 
\begin{theorem}\label{thm::1.5}
    Let $G$ be a simple connected positively LLY-curved $C_4$-free graph with $\delta(G)\geq 2$. 
    Then $G\cong C_3$, $G\cong C_5$, $G\cong T$, $G\cong H$ or $G\cong F_k$ $(k=2,3)$.
\end{theorem}

 \begin{figure}[htbp]
     \centering
     \begin{tikzpicture}[scale=1.3, x=1.00mm, y=0.7mm, inner xsep=0pt, inner ysep=0pt, outer xsep=0pt, outer ysep=0pt]
%\path[line width=0mm] (48.00,58.00) rectangle +(54.00,39.00);
\definecolor{L}{rgb}{0,0,0}
%\path[line width=0.30mm, draw=L] (75.00,90.00) ellipse (15.00mm and 4.00mm);
%\path[line width=0.30mm, draw=L] (75.00,65.00) ellipse (25.00mm and 4.00mm);
\definecolor{F}{rgb}{0,0,0}
\path[line width=0.30mm, draw=L, fill=F] (55.00,90.00) circle (1.00mm);
\path[line width=0.30mm, draw=L, fill=F] (55.00,60.00) circle (1.00mm);
%\path[line width=0.30mm, draw=L, fill=F] (75.00,90.00) circle (0.50mm);
%\path[line width=0.30mm, draw=L, fill=F] (70.00,90.00) circle (0.50mm);
%\path[line width=0.30mm, draw=L, fill=F] (80.00,90.00) circle (0.50mm);
\path[line width=0.30mm, draw=L, fill=F] (65.00,75.00) circle (1.00mm);
\path[line width=0.30mm, draw=L, fill=F] (75.00,75.00) circle (1.00mm);
%\path[line width=0.30mm, draw=L, fill=F] (75.00,65.00) circle (0.50mm);
%\path[line width=0.30mm, draw=L, fill=F] (70.00,65.00) circle (0.50mm);
%\path[line width=0.30mm, draw=L, fill=F] (80.00,65.00) circle (0.50mm);
\path[line width=0.30mm, draw=L, fill=F] (70.00,90.00) circle (1.00mm);
\path[line width=0.30mm, draw=L, fill=F] (70.00,60.00) circle (1.00mm);
\path[line width=0.30mm, draw=L, fill=F] (85.00,75.00) circle (1.00mm);
\path[line width=0.30mm, draw=L] (55.00,90.00) -- (55.00,60.00);
\path[line width=0.30mm, draw=L] (55.00,90.00) -- (65.00,75.00);
\path[line width=0.30mm, draw=L] (55.00,60.00) -- (65.00,75.00);
\path[line width=0.30mm, draw=L] (65.00,75.00) -- (75.00,75.00);
\path[line width=0.30mm, draw=L] (55.00,60.00) -- (70.00,60.00);
\path[line width=0.30mm, draw=L] (55.00,90.00) -- (70.00,90.00);
\path[line width=0.30mm, draw=L] (85.00,75.00) -- (75.00,75.00);
\path[line width=0.30mm, draw=L] (85.00,75.00) -- (70.00,90.00);
\path[line width=0.30mm, draw=L] (85.00,75.00) -- (70.00,60.00);
\end{tikzpicture}
    \caption{Graph $T$ with positive LLY curvature.}
    \label{figure T}
 \end{figure}

\begin{figure}[htbp]
    \centering
   \begin{tikzpicture}[x=2.0mm, y=2.0mm,
    inner xsep=0pt, inner ysep=0pt,
    outer xsep=0pt, outer ysep=0pt, scale=0.3]

    \definecolor{L}{rgb}{0,0,0}
    \definecolor{F}{rgb}{0,0,0}

%====================================================
% 1) 节点坐标（只调整 v4,v7,v8,v9,v10 五角星更像）
%====================================================
\coordinate (p1)  at ( 40, 48);
\coordinate (p2)  at (-40, 48);
\coordinate (p3)  at (  0, 32);

% ★ 调整后的五角星 5 点：v4,v7,v8,v9,v10
\coordinate (p4)  at (  0, 20);
\coordinate (p7)  at ( 18,  6);
\coordinate (p8)  at (-18,  6);
\coordinate (p9)  at (-10,-14);
\coordinate (p10) at ( 10,-14);

% 其余点不变
\coordinate (p5)  at (-27, 10);
\coordinate (p6)  at ( 27, 10);

\coordinate (p11) at (-52,-18);
\coordinate (p12) at ( 52,-18);
\coordinate (p13) at (-18,-26);
\coordinate (p14) at ( 18,-26);
\coordinate (p15) at (  0,-48);

%====================================================
% 2) 画边（连接关系不变）
%====================================================
\draw[line width=0.30mm, draw=L] (p1)--(p2);
\draw[line width=0.30mm, draw=L] (p1)--(p3);
\draw[line width=0.30mm, draw=L] (p1)--(p6);
\draw[line width=0.30mm, draw=L] (p1)--(p12);

\draw[line width=0.30mm, draw=L] (p2)--(p3);
\draw[line width=0.30mm, draw=L] (p2)--(p5);
\draw[line width=0.30mm, draw=L] (p2)--(p11);

\draw[line width=0.30mm, draw=L] (p3)--(p7);
\draw[line width=0.30mm, draw=L] (p3)--(p8);

\draw[line width=0.30mm, draw=L] (p4)--(p5);
\draw[line width=0.30mm, draw=L] (p4)--(p6);
\draw[line width=0.30mm, draw=L] (p4)--(p9);
\draw[line width=0.30mm, draw=L] (p4)--(p10);

\draw[line width=0.30mm, draw=L] (p5)--(p11);

\draw[line width=0.30mm, draw=L] (p6)--(p12);

\draw[line width=0.30mm, draw=L] (p7)--(p8);
\draw[line width=0.30mm, draw=L] (p7)--(p14);

\draw[line width=0.30mm, draw=L] (p8)--(p10);
\draw[line width=0.30mm, draw=L] (p8)--(p13);

\draw[line width=0.30mm, draw=L] (p9)--(p5);
\draw[line width=0.30mm, draw=L] (p9)--(p7);
\draw[line width=0.30mm, draw=L] (p9)--(p14);

\draw[line width=0.30mm, draw=L] (p10)--(p6);
\draw[line width=0.30mm, draw=L] (p10)--(p13);

\draw[line width=0.30mm, draw=L] (p11)--(p13);
\draw[line width=0.30mm, draw=L] (p11)--(p15);

\draw[line width=0.30mm, draw=L] (p12)--(p14);
\draw[line width=0.30mm, draw=L] (p12)--(p15);

\draw[line width=0.30mm, draw=L] (p13)--(p15);

\draw[line width=0.30mm, draw=L] (p14)--(p15);

%====================================================
% 3) 画点（无任何标号）
%====================================================
\path[line width=0.30mm, draw=L, fill=F] (p1)  circle (4 mm);
\path[line width=0.30mm, draw=L, fill=F] (p2)  circle (4 mm);
\path[line width=0.30mm, draw=L, fill=F] (p3)  circle (4 mm);

\path[line width=0.30mm, draw=L, fill=F] (p4)  circle (4 mm);
\path[line width=0.30mm, draw=L, fill=F] (p5)  circle (4 mm);
\path[line width=0.30mm, draw=L, fill=F] (p6)  circle (4 mm);

\path[line width=0.30mm, draw=L, fill=F] (p7)  circle (4 mm);
\path[line width=0.30mm, draw=L, fill=F] (p8)  circle (4 mm);
\path[line width=0.30mm, draw=L, fill=F] (p9)  circle (4 mm);
\path[line width=0.30mm, draw=L, fill=F] (p10) circle (4 mm);

\path[line width=0.30mm, draw=L, fill=F] (p11) circle (4 mm);
\path[line width=0.30mm, draw=L, fill=F] (p12) circle (4 mm);
\path[line width=0.30mm, draw=L, fill=F] (p13) circle (4 mm);
\path[line width=0.30mm, draw=L, fill=F] (p14) circle (4 mm);
\path[line width=0.30mm, draw=L, fill=F] (p15) circle (4 mm);

\end{tikzpicture}
    \caption{The line graph of Peterson graph denoted by $H$}
    \label{line graph of peterson graph}
\end{figure}
The graph $T$ is shown in~\cref{figure T} and the graph $H$ is shown in~\cref{line graph of peterson graph}.
Graph $H$ is also a distance regular graph with diameter three.
All of these graphs can be checked to be positively LLY-curved by the graph curvature calculator~\cite{Liu}, which is a freely accessible interactive app at 

\begin{center}
    https://www.mas.ncl.ac.uk/graph-curvature/
\end{center}

Here are some notations. 
Let $u \sim v$ denote that $u$ is adjacent to $v$.
For $v\in V(G)$, let $N(v)$ denote the neighborhood of $v$ in $G$.  
We use $N[v]$ to denote $N(v)\cup \{v\}$. 
For $u, v\in V(G)$, $d(u,v)$ denotes the number of edges in a shortest path between $u$ and $v$. 
A graph $G$ is called $F$-free if $G$ contains no $F$ as a subgraph.

\section{Preliminaries}\label{section::2}

In this section, we first recall the definition of LLY curvature.

A {\it probability distribution} over $V(G)$ is a mapping $m: V\rightarrow [0,1]$ which satisfies $\sum \limits_{x\in V(G)} m(x)=1$. Given two probability distributions $m_1,m_2$ with finite supports (which means $\{x \in V(G): m_i(x) \neq 0\}$ is a finite set), a {\it coupling} between $m_1$ and $m_2$ is a mapping $A: V \times V \rightarrow [0,1]$ with finite support such that
\begin{align*}
    \sum_{y \in V} A(x, y)=m_1(x) \quad \text { and } \quad \sum_{x \in V} A(x, y)=m_2(y).
\end{align*}
The {\it transportation distance} between two probability distributions $m_1$ and $m_2$ is defined as
\begin{align*}
    W(m_1, m_2) \coloneq \inf _A \sum _{x, y \in V} A(x, y) d(x, y),
\end{align*}
where the infimum is taken over all couplings $A$ between $m_1$ and $m_2$. 
By the Kantorovich duality theorem of a linear optimization problem, the transportation distance can also be written as
\begin{align*}
    W (m_1, m_2)=\sup _f \sum_{x \in V} f(x) [m_1(x) - m_2(x)],
\end{align*}
where the supremum is taken over all $1$-Lipschitz functions $f$, i.e.,
\begin{align*}
    |f(x) - f(y)| \leq d(x,y), \quad \forall\,x,y\in V(G).
\end{align*}
 Denote the set of all 1-Lipschitz functions $f$ by $\Lip(1)$. 
 It follows by~\cite[Remark]{LLY} that any $1$-Lipschitz function $f$ over a metric subspace can be extended to a $1$-Lipschitz function over the whole metric space. 
 Thus, $W(m_1,m_2)$ only depends on distances among vertices in $\supp(m_1)\cup \supp(m_2)$ (This is clear from the aforementioned definition) and we only need to choose the local $1$-Lipschitz functions $f$ on $\supp(m_1)\cup \supp(m_2)$.

A {\it random walk $m$} on a graph $G$ is defined as a family of probability measures ${m_v(\cdot)}_{v \in V(G)}$ such that $m_v(u) = 0$ for all $uv \notin    E(G)$. 
It follows that $m_v(u)\geq 0$ for all $u,v\in V(G)$ and $\sum\limits_{u\in N[v]}m_v(u)=1$ for all $v\in V(G)$.
Given a random walk $m$,  {\it Ollivier-Ricci curvature} is defined as
\begin{align*}
    \kappa (x, y) \coloneq 1 - \frac{W \left(m_x,    m_y\right)}{d(x, y)}.
\end{align*}
For each $x\in V(G)$, here we consider the {\it $\alpha$-lazy random walk} $m_x^\alpha(v)$ for $0\leq \alpha<1$, which is defined as
\begin{align*}
    m_x^\alpha(v)= 
    \begin{cases}
         \alpha & \text { if } v=x, \\ 
        (1-\alpha)/ d_x & \text { if } v \in N(x),\\ 
         0 & \text { otherwise.}
    \end{cases}
\end{align*}
In~\cite{LLY}, Lin, Lu and Yau modified Ollivier's definition of Ricci curvature in terms of $\alpha$-lazy random walks. 
More precisely, they defined  {\it $\alpha$-Ricci-curvature} $\kappa_\alpha(x,y)$ for any $x \neq y \in V(G)$ to be
\begin{align*}
    \kappa_\alpha(x, y)=1-\frac{W\left(m_x^\alpha, m_y^\alpha\right)}{d(x, y)}
\end{align*}
and  {\it Lin-Lu-Yau curvature} $\kappa_{\mathrm{LLY}}(x,y)$ to be
$$\kappa_{\mathrm{LLY}}(x, y)=\lim _{\alpha \rightarrow 1} \frac{\kappa_\alpha(x, y)}{1-\alpha}.$$

When $\alpha$ is large enough, the ratio $\kappa_\alpha(x,y)/(1-\alpha)$ is constant. Indeed, it was proved in~\cite{BCLMP18} that 
\begin{equation}\label{eq:bourne}
    \kappa_{\mathrm{LLY}}(x,y)=\frac{\kappa_\alpha(x,y)}{1-\alpha}, \,\,\text{for any $\alpha\in \left[\frac{1}{\max\{d_x,d_y\}+1},1\right]$}.
\end{equation} 

Although  LLY curvature $\kappa_{\mathrm{LLY}}(x,y)$ is defined for the pairs $x,y\in V(G)$, for a graph with positive LLY curvature we only need to consider $\kappa_{\mathrm{LLY}}(x,y)$ for $xy\in E(G)$ due to the following lemma.

\begin{lemma}[\cite{LLY}]

    Let $G$ be a connected graph. If $\kappa_{\mathrm{LLY}}(x,y)\geq\kappa_0$ for every edge $xy\in E(G)$, then $\kappa_{\mathrm{LLY}}(x,y)\geq\kappa_0$ for any pair of $\{x,y\}$.
\end{lemma}

For a graph $G$, the {\it combinatorial graph Laplacian} $\Delta$ is defined as
\begin{align*}
    \Delta f(x)=\frac{1}{d_x} \sum_{y \in N(x)}(f(y)-f(x)).
\end{align*}
Recently,  M\"{u}nch and Wojciechowski~\cite{Munch} proved a limit-free formulation of  LLY curvature by using the {\it graph Laplacian}, which is also a version of the Kantorovich duality theorem and plays an important role in our proof. 

\begin{theorem}[\cite{Munch}]\label{thm::2.1}

    Let $G$ be a simple graph.  
    For $x\neq y\in V(G)$, there holds
    \begin{align*}
        \kappa_{\mathrm{LLY}}(x, y)=\inf _{\substack{f \in \Lip(1) \\ \nabla_{y x} f=1}} \nabla_{x y} \Delta f,
    \end{align*}
    where $\nabla_{x y} f=\frac{f(x)-f(y)}{d(x, y)}$ is the gradient function.
\end{theorem}
%Theorem \ref{thm::2.1} plays an important role in our proof.

\section{Proofs of Theorems \ref{thm::1.4} and \ref{thm::1.5}}\label{section::3}
First, we recall two lemmas given by Lin, Lu, and Yau~\cite{girth5}.

\begin{lemma}[\cite{girth5}]\label{lemma::3.1}

    Suppose that an edge $xy$ in a graph $G$ is not in any $C_3$, $C_4$, or $C_5$. 
    Then $\kappa_{\mathrm{LLY}}(x,y) =\frac{2}{d_x}+\frac{2}{d_y}-2$.
\end{lemma}

\begin{lemma}[\cite{girth5}]\label{lemma::3.2}

    Suppose that an edge $xy$ in a graph $G$ is not in any $C_3$ or $C_4$. 
    Then $\kappa_{\mathrm{LLY}}(x,y)\leq\frac{1}{d_x}+\frac{2}{d_y}-1$.
\end{lemma}

Next, we prove a lemma about the local structure of our graphs of interest.

 \begin{lemma}\label{lemma::3.3}
 
     Suppose that $G$ is a simple connected positively LLY-curved $C_4$-free graph. 
     For $xy\in E(G)$, if $d_x\geq 4$ and $d_y\geq 2$, then there  exists exactly one triangle containing the edge $xy$. 
     Moreover, if $\delta(G)\geq 2$, then $d_x$ must be an even integer for any $x$ satisfying $d_x\geq 4$.
 \end{lemma}
 
 \begin{proof}
 
      Let $xy\in E(G)$ such that $d_x\geq 4$ and $d_y\geq 2$. Suppose by contradiction that there is no triangle containing $xy$.
      It follows by \cref{lemma::3.2} that
     \begin{align*}
        \kappa_{\mathrm{LLY}}(x,y)&\leq \frac{2}{d_x}+\frac{1}{d_y}-1\\
        &\leq \frac{1}{2}+\frac{1}{2}-1
        =0,
    \end{align*}
    which contradicts the assumption. 

    Since each edge incident to $x$ is contained in a unique triangle (Otherwise, we would find a $C_4$), the degree of $x$ must be even when $d_x\geq 4$.
    Thus the lemma follows.
 \end{proof}
 
For convenience of the calculation in our proof, here we expand $\nabla_{xy} \Delta f$ with the definitions of the graph Laplacian and gradient function as follows.
 \begin{align*}
     \nabla_{xy} \Delta f
    & = \Delta f(x)-\Delta f(y)\\
    &=\frac{1}{d_x}\sum\limits_{u\in N(x)}(f(u)-f(x))-\frac{1}{d_y}\sum\limits_{v\in N(y)}(f(v)-f(y)).
\end{align*}
 
\begin{proof}[Proof of \cref{thm::1.4}]

Let $xy\in E(G)$ such that $d_x=\Delta(G)$. 
By \cref{lemma::3.3}, $\Delta(G)$ is an even integer when $\Delta(G)\geq 4$. 
Thus, $\Delta(G)\neq 7$.

Assume that $\Delta(G)\geq 8$. 
Owing to \cref{lemma::3.3}, there exists a triangle denoted by $xyw$ that contains $xy$, where $w$ is the only common neighbor of $x$ and $y$ (Otherwise we can find a $C_4$).

Consider a function $f:N(x)\cup N(y)\rightarrow\mathbb{R}$ given by
\begin{align*}
    f(z)= 
    \begin{cases}
        -1 & \text { if } z \in N(x)\setminus \{y,w\} ; \\
         0 & \text { if } z=x ; \\
         1 & \text { if } z \in N[y]\setminus\{x\}.
    \end{cases}
\end{align*}
  Let $u\in N(x)\setminus \{y,w\}$ and $v\in N[y]\setminus \{x,w\}$. 
  Since $G$ is $C_4$-free, $d(u,v)\geq 2$ for any $u$ and $v$. Therefore, $f\in \Lip(1)$. 
  $f(x)=0$ and $f(y)=1$ lead to $\nabla_{yx} f=\frac{f(y)-f(x)}{d(y, x)}=1$.
By the assumption and \cref{thm::2.1}, we have
\begin{align*}
    0< \kappa_{\mathrm{LLY}}(x,y)
    &\leq\frac{1}{d_x}\sum\limits_{u\in N(x)}(f(u)-f(x))-\frac{1}{d_y}\sum\limits_{v\in N(y)}(f(v)-f(y))\\
    & = \frac{1}{d_x}(-1\cdot(d_x-2)+1\cdot2-0\cdot d_x)-\frac{1}{d_y}(1\cdot(d_y-1)+0\cdot1-1\cdot d_y)   \\
    & = \frac{4}{d_x}+\frac{1}{d_y}-1 \\
    & \leq \frac{1}{2}+\frac{1}{2}-1=0,
\end{align*}
which is a contradiction. 
Hence, $\Delta(G) \leq 6$. Moreover, we find that the friendship graph $F_3$ is positively LLY-curved, which has maximum degree  6.

Therefore, the proof is complete.
\end{proof}

Actually, if we allow the existence of pendant edges, then $\Delta(G)=7$ can be realized. 
The following corollary tells us that the addition of one pendant edge can still make the graph's LLY curvature  positive.

\begin{corollary}

    Let $G$ be a simple connected positively LLY-curved $C_4$-free graph with only one pendant edge. 
    Then $\Delta(G)\leq 7$ and the upper bound is sharp.
\end{corollary}

\begin{proof}

    Let $xy\in E(G)$ such that $d_x=\Delta(G)$ and $d_y\geq 2$. 
    The same argument as  the proof of  \cref{thm::1.4} furnishes $\Delta(G)\leq 7$. 
    On the other hand, we can find a positively LLY-curved graph $F_3'$ shown in \cref{F_3'} whose maximum degree is $7$ and has one pendant edge.  Its LLY curvature can be checked by the graph curvature calculator.
\end{proof}

\begin{figure}[htbp]
    \centering
    %\hspace*{-5cm}% 往左移 1cm ，负数表示左移，可调整数值
   \begin{tikzpicture}[scale=1.6,every node/.style={draw, circle, fill=black, inner sep=1mm}]

% 定义中心顶点
\node (O) at (0,0) {};

\foreach \angle in {300} {
\node (x) at (\angle+30:1){};
}

% 定义三个三角形的顶点
\foreach \angle in {0,120,240} {
    \node (A\angle) at (\angle:1) {};
    \node (B\angle) at (\angle+60:1) {};
}

% 连接三个三角形
\foreach \angle in {0,120,240} {
    \draw[line width=0.3mm] (O) -- (A\angle);
    \draw[line width=0.3mm] (O) -- (B\angle);
    \draw[line width=0.3mm] (A\angle) -- (B\angle);
    \draw[line width=0.3mm] (O) -- (x);
}

\end{tikzpicture}
     \caption{Graph $F_3'$ with positive LLY curvature and one pendant edge.}
   \label{F_3'}
\end{figure}

In view of \cref{thm::1.4}, we need more information to prove \cref{thm::1.5}. 
We provide two lemmas which might be more convenient than \cref{lemma::3.1} and \cref{lemma::3.2} when we apply them in our proof.

\begin{lemma}\label{lemma3.4}

   Suppose that $G$ is positively LLY-curved and $C_4$-free with $\delta(G)\geq 2$.  
   If $xy\in E(G)$ is  not in any $C_3$, then $xy$ is contained in a $C_5$.
\end{lemma}

 \begin{proof}
 
    Choose an edge $xy \in E(G)$ which is not in any $C_3$.
    If  $xy$ is not contained in any $C_5$, then by \cref{lemma::3.1} we have
    \begin{align*}
        0 < \kappa_{\mathrm{LLY}}(x,y)=\frac{2}{d_x}+\frac{2}{d_y}-2\leq 0,
    \end{align*}
    which is a contradiction.
\end{proof}

\begin{lemma}\label{lemma3.5}

    Suppose that $G$ is  positively LLY-curved and $C_4$-free  with $\delta(G)\geq 2$. 
    If $xy\in E(G)$ with $d_x=3$ is  not in any $C_3$, then $d_y=2$.
\end{lemma}

 \begin{proof}

    Choose an edge $xy \in E(G)$ such that $xy$ is not in any $C_3$ and $d_x=3$.
    Assume that $d_y\geq3$. 
    \cref{lemma::3.2} tells us 
    \begin{align*}
        0 < \kappa_{\mathrm{LLY}}(x,y)\leq \frac{1}{3}+\frac{2}{3}-1=0,
    \end{align*}
    which leads to a contradiction. 
    Hence, $d_y=2$.
\end{proof}

\begin{proof}[Proof of~\cref{thm::1.5}] 

Due to \cref{lemma::3.3}, $\Delta(G)$ is an even integer when $\Delta(G)\geq 4$.
Then we shall proceed with the proof by the following four cases.

    \begin{case}
    
        $\Delta(G)=2$. 
        
        In the this, $G$ can only be isomorphic to $C_3$ with constant LLY curvature $\frac{3}{2}$ or $C_5$ with constant LLY curvature $\frac{1}{2}$, as for $C_n$ ($n\geq 6)$, it has constant LLY curvature $0$ (\cite{LLY}, Example 2).

    \end{case}

    \begin{case}
    
        $\Delta(G)=3$.

        \begin{claim}\label{claim::3}
            $G$ contains a triangle.
        \end{claim}
\begin{proof}[Proof of \cref{claim::3}]

    Suppose that $G$ contains no triangle. Consider an edge $xy\in E(G)$ such that $d_x=3$. Due to \cref{lemma3.4}, there exists a 5-cycle containing $xy$, denoted by $xyvtu$. 
    Let $w$ be another neighbor of $x$. 
      According to \cref{lemma3.5}, we have $d_y=d_u=d_w=2$.
    %The local structure is shown in Figure \ref{C3-free1}. 
 Moreover, at most one of $t$ or $v$ can have degree 3.
  If $d_t=d_v=2$, then no cycle contains the edge $xw$, which  contradicts \cref{lemma3.4}.
    Hence, exactly one of $t$ or $v$ has degree 3. 
    WLOG, let $d_v=3$.

 %   \begin{figure}[h]
  %  \centering
   % \includegraphics[width=0.4\textwidth,height=0.24\textwidth]{C3free1.png}
    %\caption{Local structure of $G$ when $G$ is triangle-free.}
    %\label{C3-free1}
%\end{figure}

Since $G$ is $C_4$-free, $v$ is not adjacent to $w$ and the vertices $x, u, t, v, w$ cannot form a $C_5$. 
In fact, since the edge $xw$ must lie in a 5-cycle by \cref{lemma3.4},
$v$ and $w$ must share a common neighbor, denoted by $p$. 
Owing to \cref{lemma3.5}, we have $d_p=2$.
Thus $G\cong G_1$, as shown in~\cref{c_5p}.
\begin{figure}[htbp]
   \centering
    \begin{tikzpicture}[scale=1.6,every node/.style={draw, circle, fill=black, inner sep=1mm}]
\node (O) at (0,0) [label=left:{$y$}]{};
\node (a) at (-0.6,0.9) [label=left:{$x$}]{};
\node (b) at (-1.2,0.1) [label=left:{$w$}]{};
\node (c) at (-1.1,-0.75) [label=left:{$p$}]{};
\node (d) at (0,-1) [label=below:{$v$}]{};
\node (e) at (1,-0.7) [label=right:{$t$}]{};
\node (f) at (1,0.2) [label=right:{$u$}]{};
    \draw[line width=0.3mm] (O) -- (a);
   \draw[line width=0.3mm] (O) -- (d);
    \draw[line width=0.3mm] (a) -- (b);
 \draw[line width=0.3mm] (a) -- (f);
\draw[line width=0.3mm] (b) -- (c);
    \draw[line width=0.3mm] (c) -- (d);
    \draw[line width=0.3mm] (d) -- (e);
     \draw[line width=0.3mm] (e) -- (f);
\end{tikzpicture}
     \caption{Graph $G_1$ containing some edges with 0 LLY curvature.}
    \label{c_5p}
\end{figure}
\end{proof}

Based on \cref{claim::3}, we obtain a triangle $xyu$ in $G$ such that $d_x=3$. 
Let $w$ be another neighbor of $x$. Since $G$ is $C_4$-free, the edge $xw$ does not lie in any $C_3$. 
    Due to \cref{lemma3.5}, we have $d_w=2$. 
    Let $v$ be the other neighbor of $w$. Note that $v$ is not adjacent to either $y$ or $u$; otherwise a $C_4$ would be formed. 
    
    Suppose that $d(u,v)=3$ (or $d(y,v)=3$). 
    Define a  function $f_1: N(x)\cup N(w)\rightarrow \mathbb{R}$ by
   \begin{align*}
        f_1(z)=
        \begin{cases}
            -1 & \text { if } z=u ; \\
             0 & \text { if } z \in \{x,y\}; \\
             1 & \text { if } z=w;\\
             2 & \text { if } z=v.
        \end{cases}
   \end{align*}
Since $G$ is $C_4$-free and $d_x=3$, we have $f_1\in \Lip(1)$ with $\nabla_{wx} f_1=1$. Now Theorem \ref{thm::2.1} yields 
\begin{align*}
     0 < \kappa_{\mathrm{LLY}}(x,w) & \leq \Delta f_1(x)-\Delta f_1(w)\\
   & \leq \frac{1}{3}(0-1+1-0) - \frac{1}{2}(2+0-2)
 =0.
\end{align*}
Thus, $d(y,v)=d(u,v)=2$, which implies that there exist two vertices $w_1$ and $w_2$ that $w_1$ is the unique common neighbor of $u$ and $v$, and $w_2$ is the unique common neighbor of $y$ and $v$. 
Note that $w_1\neq w_2$; otherwise, the vertices $x, u, w_1, y$ would form a 4-cycle.
The same argument applied to $d_w$ also yields $d_{w_1}=d_{w_2}=2$, which implies $G$ is isomorphic to the graph $T$ shown in~\cref{figure T}.

    \end{case}

    \begin{case}
    
        $\Delta(G)=4$. 
        
        By \cref{lemma::3.3}, $F_2$ is a subgraph of $G$. 
        We now divide this case into two subcases.

    \begin{subcase}
        The graph $F$ shown in \cref{3C3} is a subgraph of $G$.

\begin{figure}[htbp]
    \centering
    %\hspace*{-5cm}% 往左移 1cm ，负数表示左移，可调整数值
   \begin{tikzpicture}[
    scale=1.6,
    vertex/.style={
        draw=black,
        fill=black,
        circle,
        inner sep=1mm,
        outer sep=0pt
    },
]

% 定义顶点并添加标签
\node[vertex, label={[label]below :$x$}] (x) at (-0.5,0){};
\node[vertex, label={[label]below :$y$}] (y) at (0.5,0) {};
\node[vertex, label={[label]below:$u_1$}] (u_1) at (-1.5,0) {};
\node[vertex, label={[label]above :$u_2$}] (u_2) at (-1,0.9) {};
\node[vertex, label={[label]below:$v_1$}] (v_1) at (1.5,0) {};
\node[vertex, label={[label]above :$v_2$}] (v_2) at (1,0.9) {};
\node[vertex, label={[label]above:$w$}] (w) at (0,0.9) {};

% 绘制连接线
\draw[line width=0.3mm] (x) -- (y);
\draw[line width=0.3mm] (x) -- (w);
\draw[line width=0.3mm] (y) -- (w);
\draw[line width=0.3mm] (x) -- (u_1);
\draw[line width=0.3mm] (x) -- (u_2);
\draw[line width=0.3mm] (y) -- (v_1);
\draw[line width=0.3mm] (y) -- (v_2);
\draw[line width=0.3mm] (v_1) -- (v_2);
\draw[line width=0.3mm] (u_1) -- (u_2);

\end{tikzpicture}
      \caption{Graph $F$ with three triangles.}
                \label{3C3}
\end{figure}   
  Assume that $F$ occurs as a subgraph of $G$ with vertices labeled as~\cref{3C3}. 
    Note that no $u_i$ is adjacent to any $v_j$, and none of $u_i$ or $v_j$ is adjacent to $w$; otherwise, we would find a $C_4$.

    When $d(u_i,v_j)=3$ for each $i,j$,  we consider a  function $f_2 : N(x) \cup N(y) \rightarrow \mathbb{R}$ given by
    \begin{align*}
        f_2(z)= 
        \begin{cases}
            -1 & \text { if } z\in \{u_1,u_2\} ; \\
             0 & \text { if } z\in\{x,w\}; \\
             1 & \text{ if } z=y;\\
             2 & \text{ if } z\in \{v_1,v_2\}.
        \end{cases}
    \end{align*}
Since $G$ is $C_4$-free, we have $f_2\in \Lip(1)$ with $\nabla_{yx} f_2=1$. 
By \cref{thm::2.1}, we have
\begin{align*}
    0 < \kappa_{\mathrm{LLY}} (x,y) & \leq \Delta f_2(x)-\Delta f_2(y)\\
    & \leq \frac{1}{4} (-1-1+1+0-0) - \frac{1}{4} (0+0+2+2-4)
    =- \frac{1}{4},
\end{align*}
which leads to a contradiction.
Therefore, $d(u_i,v_j)=2$ holds for some $i$ and $j$.
WLOG, assume $d(u_1,v_1)=2$, and let $p$ be the unique common neighbor of $u_1$ and $v_1$.
We consider the following two cases. 

The  first case is $d_{u_1}=d_{v_1}=3$. 
Define a function $f_3:N(u_1) \cup N(p) \rightarrow \mathbb{R}$ by
\begin{align*}
    f_3(z)= 
    \begin{cases}
        -1 & \text{ if } z=u_2 ; \\
         0 & \text{ if } z \in \{u_1,x\}; \\
         1 & \text{ if } z\in N[p] \setminus \{v_1,u_1\}; \\
         2 & \text{ if } z=v_1.
    \end{cases}
\end{align*}
Since $G$ is $C_4$-free and $d_{u_1}=d_{v_1}=3$, we have $f_3\in \Lip(1)$ and $\nabla_{pu_1} f_3=1$. 
Thus, Theorem \ref{thm::2.1} yields 
\begin{align*}
      0 < \kappa_{\mathrm{LLY}}(u_1,p) & \leq \Delta f_3(u_1)-\Delta f_3(p)\\
    &\leq \frac{1}{3}(-1+1+0-0)-\frac{1}{d_p}(d_p-2+2+0-d_p)
    =0.
\end{align*}

The second case is $d_{u_1}=4$ (or $d_{v_1}=4$).
It follows from \cref{lemma::3.3} that $p_0,p,u_1$ must form a triangle, where $p_0$ denotes another neighbor of $u_1$.
Note also that $d_p=d_{v_1}$.
If $d_p=d_{v_1}=3$, then the edge $p v_1$ lies in no $C_3$ or $C_4$, contradicting \cref{lemma3.5}.
For $d_p=d_{v_1}=4$, since $G$ is $C_4$-free, $p_0$ is not adjacent to $v_1$ and $v_2$.
Denote another neighbor of $v_1$ by $p_1$, which is also a neighbor of $p$ due to~\cref{lemma::3.3}.
We consider two situations as follows.
\begin{itemize}
    \item  $d(p_0,v_2) \geq 3$.
    
    We define a  function $f_4:N(p)\cup N(v_1)\rightarrow \mathbb{R}$  by
\begin{align*}
    f_4(z)= 
    \begin{cases}
        -1 & \text { if } z\in \{u_1,p_0\} ; \\
         0 & \text { if } z=p; \\
         1 & \text { if } z\in \{p_1,y\}; \\
         2&\text{ if } z=v_2.
    \end{cases}
\end{align*}
Since $G$ is $C_4$-free and $d_{u_1}=\Delta(G)$, it satisfies that $f_4\in \Lip(1)$ with $\nabla_{v_1p} f_4=1$. 
Applying~\cref{thm::2.1}, we have
\begin{align*}
     0<\kappa_{\mathrm{LLY}}(p,v_1) & \leq \Delta f_4(p)-\Delta f_4(v_1)\\
     & \leq \frac{1}{4}(-1-1+1+1-0) - \frac{1}{4} (1+1+0+2-4) = 0.
\end{align*}
Hence $d(u_i,v_j)\neq 2$ for any $i,j$, which contradicts $d(u_i,v_j)=2$ for some $i$ and $j$.

\item $d(p_0,v_2)=2$.

Now we have five vertices $x,y,u_1,v_1,p$ with maximum degree.
As $G$ is $C_4$-free, $p_0$ is not adjacent to $u_2,w,v_2,p_1$.
Thus, we can find a new vertex $a$ such that $a\sim p_0$ and $a\sim v_2$.
The condition $C_4$-free also leads that $a$ is not adjacent to $u_2,w,p_1$.
Next we claim that $d(a,w)=2$.
If $d(a,w)\geq 3$, then we can construct a 1-Lipschitz function $f_5: N(y)\cup N(v_2)\rightarrow \mathbb{R}$ such that $\nabla_{v_2y}f_5=1$ as follows.
\begin{align*}
    f_5(z)= 
    \begin{cases}
        -1 & \text { if } z\in \{x,w\} ; \\
         0 & \text { if } z=y; \\
         1 & \text { if } z\in N[v_2]\setminus\{a,y\}; \\
         2&\text{ if } z=a.
    \end{cases}
\end{align*}
Owing to~\cref{thm::2.1}, we have
\begin{align*}
     0<\kappa_{\mathrm{LLY}}(y,v_2) & \leq \Delta f_5(y)-\Delta f_5(v_2)\\
     & \leq \frac{1}{4}(-1-1+1+1-0) - \frac{1}{d_{v_2}} ((d_{v_2}-2)+2-d_{v_2}) = 0.
\end{align*}
Therefore, $d(a,w)=2$.
The condition $C_4$-free implies that the unique common neighbor of $a$ and $w$ is a new vertex, denoted by $b$.
If $d_{p_0}=3$, then $p_0$ only has three neighbors $a,p,u_1$, which indicates that the edge $ap_0$ is not in any $C_3$.
It is impossible according to~\cref{lemma3.5}.
It follows that $d_{p_0}=4$.
If $b\not\sim p_0$, then we can find a new vertex $c$ such that $a,c,p_0$ form a triangle by~\cref{lemma::3.3}, which leads that $b\sim v_2$, contradicting $C_4$-free.
So $b\sim p_0$, and $a,b,p_0$ form a $C_3$. 
Combining with the fact above, \cref{lemma::3.3} also tells us $d_a=d_{v_2}$.
If $d_a=d_{v_2}=3$, then $av_2$ is not in any $C_3$, which contradicts~\cref{lemma3.5}.
Therefore, $d_a=d_{v_2}=4$, and they have a unique common neighbor, denoted by $a_0$.
Now we have eight vertices $x,y,u_1,v_1,p,p_0,a,v_2$ with maximum degree.
Notice that $a_0$ cannot be adjacent to $b,w,p_1$ due to $C_4$-free.
Apply the similar argument as above, and we obtain that $d_b=d_w=4$.
The unique common neighbor of $b$ and $w$ is denoted by $b_0$.
We claim that $b_0\sim p_1$.
If $b_0\not\sim p_1$, then we can construct a function similar to $f_5$ to find a contradiction. 
Define a  function $f_6:N(p_0)\cup N(p)\rightarrow \mathbb{R}$  by
\begin{align*}
    f_6(z)= 
    \begin{cases}
        -1 & \text { if } z\in \{a,b\} ; \\
         0 & \text { if } z=p_0; \\
         1 & \text { if } z\in \{p,v_1,u\}; \\
         2&\text{ if } z=p_1.
    \end{cases}
\end{align*}
Since $p_1$ is not adjacent to $v_2,a_0,b,p_0,a,w,b_0$, we have $d(p_1,b)=d(p_1,a)=3$.
It is clear that $f_6\in\Lip(1)$ and $\nabla_{pp_0}f_6=1$. 
By~\cref{thm::2.1}, $\kappa_{\mathrm{LLY}}(p_0,p)\leq 0$, a contradiction.
As the argument before, $d_{b_0}=d_{p_1}=4$.
The unique common neighbor of $b_0$ and $p_1$ is denoted by $p'$.
If $u_2\not\sim a_0$, then we  consider a function $f_7: N(a)\cup N(p_0) \rightarrow \mathbb{R}$ as
\begin{align*}
    f_7(z)= 
    \begin{cases}
        -1 & \text { if } z\in \{a_0,v_2\} ; \\
         0 & \text { if } z=a; \\
         1 & \text { if } z\in \{p_0,p,b\}; \\
         2&\text{ if } z=u_1.
    \end{cases}
\end{align*}
Since $a_0$ and $v_2$ are not adjacent to the neighbors of $u_1$, $d(u_1,a_0)=d(u_1,v_2)=3$.
Thus, $f_7\in\Lip(1)$ and $\nabla_{p_0a}f_7=1$.
By~\cref{thm::2.1}, $\kappa_{\mathrm{LLY}}(a,p_0)\leq 0$, a contradiction.
So $a_0\sim u_2$.
As the argument before, $d_{u_2}=d_{a_0}=4$.
Last, we prove that $a_0$ and $u_2$ must be adjacent to $p'$.
If $p'$ is not the common neighbor of $a_0$ and $u_2$, then there exists a new vertex $e$ such that $a_0,u_2,e$ form a $C_3$ by~\cref{lemma::3.3}.
Define a function $f_8: N(u_1)\cup N(p) \rightarrow \mathbb{R}$ by
\begin{align*}
    f_8(z)= 
    \begin{cases}
        -1 & \text { if } z\in \{x,u_2\} ; \\
         0 & \text { if } z=u_1; \\
         1 & \text { if } z\in \{p_0,p,v_1\}; \\
         2&\text{ if } z=p_1.
    \end{cases}
\end{align*}
Clearly, $p_1$ is not adjacent to the neighbors of $u_2$ and $x$.
So $d(p_1,u_2)=d(p_1,x)=3$.
It follows that $f_8\in\Lip(1)$ and $\nabla_{pu_1}f_8=1$.
By~\cref{thm::2.1}, $\kappa_{\mathrm{LLY}}(u_1,p)\leq 0$, a contradiction.
Therefore, such a vertex $e$ cannot exist, and $u_2,a_0$ must be adjacent to $p'$.
Now we get a 4-regular graph with 15 vertices.
Actually, this graph is the line graph of Peterson graph shown as~\cref{line graph of peterson graph}.
For clarity, we label all the vertices that we have used above in~\cref{labeled H}.
\begin{figure}[htbp]
    \centering
   \begin{tikzpicture}[x=2.0mm, y=2.0mm,
    inner xsep=0pt, inner ysep=0pt,
    outer xsep=0pt, outer ysep=0pt, scale=0.4]

    \definecolor{L}{rgb}{0,0,0}
    \definecolor{F}{rgb}{0,0,0}

%====================================================
% 1) 节点坐标（只调整 v4,v7,v8,v9,v10 五角星更像）
%====================================================
\coordinate (p1)  at ( 40, 48) ;
\coordinate (p2)  at (-40, 48);
\coordinate (p3)  at (  0, 32);

% ★ 调整后的五角星 5 点：v4,v7,v8,v9,v10
\coordinate (p4)  at (  0, 20);
\coordinate (p7)  at ( 18,  6);
\coordinate (p8)  at (-18,  6);
\coordinate (p9)  at (-10,-14);
\coordinate (p10) at ( 10,-14);

% 其余点不变
\coordinate (p5)  at (-27, 10);
\coordinate (p6)  at ( 27, 10);

\coordinate (p11) at (-52,-18);
\coordinate (p12) at ( 52,-18);
\coordinate (p13) at (-18,-26);
\coordinate (p14) at ( 18,-26);
\coordinate (p15) at (  0,-48);

%====================================================
% 2) 画边（连接关系不变）
%====================================================
\draw[line width=0.30mm, draw=L] (p1)--(p2);
\draw[line width=0.30mm, draw=L] (p1)--(p3);
\draw[line width=0.30mm, draw=L] (p1)--(p6);
\draw[line width=0.30mm, draw=L] (p1)--(p12);

\draw[line width=0.30mm, draw=L] (p2)--(p3);
\draw[line width=0.30mm, draw=L] (p2)--(p5);
\draw[line width=0.30mm, draw=L] (p2)--(p11);

\draw[line width=0.30mm, draw=L] (p3)--(p7);
\draw[line width=0.30mm, draw=L] (p3)--(p8);

\draw[line width=0.30mm, draw=L] (p4)--(p5);
\draw[line width=0.30mm, draw=L] (p4)--(p6);
\draw[line width=0.30mm, draw=L] (p4)--(p9);
\draw[line width=0.30mm, draw=L] (p4)--(p10);

\draw[line width=0.30mm, draw=L] (p5)--(p11);

\draw[line width=0.30mm, draw=L] (p6)--(p12);

\draw[line width=0.30mm, draw=L] (p7)--(p8);
\draw[line width=0.30mm, draw=L] (p7)--(p14);

\draw[line width=0.30mm, draw=L] (p8)--(p10);
\draw[line width=0.30mm, draw=L] (p8)--(p13);

\draw[line width=0.30mm, draw=L] (p9)--(p5);
\draw[line width=0.30mm, draw=L] (p9)--(p7);
\draw[line width=0.30mm, draw=L] (p9)--(p14);

\draw[line width=0.30mm, draw=L] (p10)--(p6);
\draw[line width=0.30mm, draw=L] (p10)--(p13);

\draw[line width=0.30mm, draw=L] (p11)--(p13);
\draw[line width=0.30mm, draw=L] (p11)--(p15);

\draw[line width=0.30mm, draw=L] (p12)--(p14);
\draw[line width=0.30mm, draw=L] (p12)--(p15);

\draw[line width=0.30mm, draw=L] (p13)--(p15);

\draw[line width=0.30mm, draw=L] (p14)--(p15);

%====================================================
% 3) 画点（无任何标号）
%====================================================
\path[line width=0.30mm, draw=L, fill=F] (p1)  circle (2 mm);
\path[line width=0.30mm, draw=L, fill=F] (p2)  circle (2 mm);
\path[line width=0.30mm, draw=L, fill=F] (p3)  circle (2 mm);

\path[line width=0.30mm, draw=L, fill=F] (p4)  circle (2 mm);
\path[line width=0.30mm, draw=L, fill=F] (p5)  circle (2 mm);
\path[line width=0.30mm, draw=L, fill=F] (p6)  circle (2 mm);

\path[line width=0.30mm, draw=L, fill=F] (p7)  circle (2 mm);
\path[line width=0.30mm, draw=L, fill=F] (p8)  circle (2 mm);
\path[line width=0.30mm, draw=L, fill=F] (p9)  circle (2 mm);
\path[line width=0.30mm, draw=L, fill=F] (p10) circle (2 mm);

\path[line width=0.30mm, draw=L, fill=F] (p11) circle (2 mm);
\path[line width=0.30mm, draw=L, fill=F] (p12) circle (2 mm);
\path[line width=0.30mm, draw=L, fill=F] (p13) circle (2 mm);
\path[line width=0.30mm, draw=L, fill=F] (p14) circle (2 mm);
\path[line width=0.30mm, draw=L, fill=F] (p15) circle (2 mm);

\node[font=\large, draw=none, fill=none, inner sep=2mm] at (p1) [above right] {$p_1$};
\node[font=\large, draw=none, fill=none, inner sep=2mm] at (p2) [above left]  {$p'$};
\node[font=\large, draw=none, fill=none,inner sep=2mm] at (p3) [above]       {$b_0$};
\node[font=\large, draw=none, fill=none,inner sep=2mm] at (p4) [above]       {$v_2$};
\node[font=\large, draw=none, fill=none, inner sep=2mm] at (p5) [left]        {$a_0$};
\node[font=\large, draw=none, fill=none, inner sep=2mm] at (p6) [right]       {$v_1$};
\node[font=\large, draw=none, fill=none, inner sep=2mm] at (p7) [right]       {$b$};
\node[font=\large, draw=none, fill=none, inner sep=2mm] at (p8) [left]        {$w$};
\node[font=\large, draw=none, fill=none, inner sep=2mm] at (p9) [left]        {$a$};
\node[font=\large, draw=none, fill=none, inner sep=2mm] at (p10)[right]       {$y$};
\node[font=\large, draw=none, fill=none, inner sep=2mm] at (p11)[left]        {$u_2$};
\node[font=\large, draw=none, fill=none, inner sep=2mm] at (p12)[right]       {$p$};
\node[font=\large, draw=none, fill=none, inner sep=2mm] at (p13)[below]        {$x$};
\node[font=\large, draw=none, fill=none, inner sep=2mm] at (p14)[below]       {$p_0$};
\node[font=\large, draw=none, fill=none, inner sep=2mm] at (p15)[below]       {$u_1$};
\end{tikzpicture}
    \caption{Labeled graph $H$}
    \label{labeled H}
\end{figure}
\end{itemize}    
    \end{subcase}

           \begin{subcase}
        The graph $F$ is not a subgraph of $G$. 

        Choose a vertex $x$ with $d_x=4$. 
    If $G\ncong F_2$, then there exists a neighbor $y$ of $x$ such that $d_y=3$ (Note that if $d_y=4$, then the subgraph $F$ would appear). 
   Let $x,u_1,u_2$ form the other triangle. 
   Denote by $w$ the common neighbor of $x$ and $y$, and by $v$ another neighbor of $y$. 
   Owing to \cref{lemma3.5}, we have $d_v=2$, and let $v_0$ be the other neighbor of $v$. 
    The local structure of $G$ is shown as \cref{2C3}. 

\begin{figure}[htbp]
    \centering
    %\hspace*{-5cm}% 往左移 1cm ，负数表示左移，可调整数值
    \begin{tikzpicture}[
    scale=1.6,
    vertex/.style={
        draw=black,
        fill=black,
        circle,
        inner sep=1mm,
        outer sep=0pt
    },
]

% 定义顶点并添加标签
\node[vertex, label={[label]below :$x$}] (x) at (-0.5,0){};
\node[vertex, label={[label]below :$y$}] (y) at (0.5,0) {};
\node[vertex, label={[label]below:$u_1$}] (u_1) at (-1.5,0) {};
\node[vertex, label={[label]above :$u_2$}] (u_2) at (-1,0.9) {};
\node[vertex, label={[label]below:$v$}] (v) at (1.5,0) {};
\node[vertex, label={[label]above :$v_0$}] (v_0) at (1,0.9) {};
\node[vertex, label={[label]above:$w$}] (w) at (0,0.9) {};

% 绘制连接线
\draw[line width=0.3mm] (x) -- (y);
\draw[line width=0.3mm] (x) -- (w);
\draw[line width=0.3mm] (y) -- (w);
\draw[line width=0.3mm] (x) -- (u_1);
\draw[line width=0.3mm] (x) -- (u_2);
\draw[line width=0.3mm] (y) -- (v);
\draw[line width=0.3mm] (v) -- (v_0);
\draw[line width=0.3mm] (u_1) -- (u_2);

\end{tikzpicture}
     \caption{Local structure of $G$ with subgraph $F_2$.}
    \label{2C3}
\end{figure}   

By~\cref{lemma3.4}, $yv$ must lie in some $C_5$.
We divide it into two cases.
\begin{itemize}
    \item  $yv,vv_0,wy$ are in a common $C_5$.
    
Since $G$ is $C_4$-free, we can find a new vertex $w_0$ such that $y,v,v_0,w_0,w$ form a $C_5$.
By~\cref{lemma3.5}, $d_{w_0}=2$.
If $v_0\sim u_1$ (or $v_0\sim u_2$), then $d_{u_1}=d_{v_0}=3$ and $u_2v_0$ is not in any $C_3$, which are in contradiction with~\cref{lemma3.5}.
If $v_0$ is not adjacent to $u_1$ and $u_2$, then $yv$ would have non-positive LLY curvature.
Define a  function $f_9:N(y)\cup N(v)\rightarrow \mathbb{R}$ by
\begin{align*}
    f_9(z)= 
    \begin{cases}
        -1 & \text { if } z=x ; \\
         0 & \text { if } z \in \{y,w\}; \\
         1 & \text { if } z=v; \\
         2 & \text { if } z=v_0.
  \end{cases}
\end{align*}
Clearly, $f_9\in \Lip(1)$ and $\nabla_{vy} f_9=1$ because $G$ is $C_4$-free. 
By~\cref{thm::2.1}, we have $\kappa_{\mathrm{LLY}}(y,v)\leq 0$.

    \item $yv,vv_0,xy$ are in a common $C_5$.

    In this case, $v_0$ must be adjacent to exactly one of $u_1$ and $u_2$.
    WLOG, assume that $v_0\sim u_2$.
     By~\cref{lemma3.5}, we have $d_{v_0}=2$ and $d_{u_2}=3$.
     Consider a  function $f_{10}:N(y)\cup N(v)\rightarrow \mathbb{R}$ given by
\begin{align*}
    f_{10}(z)= 
    \begin{cases}
        -1 & \text { if } z=w ; \\
         0 & \text { if } z \in \{y,x\}; \\
         1 & \text { if } z=v; \\
         2 & \text { if } z=v_0.
  \end{cases}
\end{align*}
Again, $f_{10}\in \Lip(1)$ and $\nabla_{vy} f_{10}=1$. 
By~\cref{thm::2.1}, we have $\kappa_{\mathrm{LLY}}(y,v)\leq 0$, which leads to a contradiction. 
\end{itemize}

Therefore, $G\cong F_2$.
    \end{subcase}

    \end{case}

    \begin{case}
    
        $\Delta(G)=6$.
        
        According to \cref{lemma::3.3}, $G$ must contain $F_3$ as a subgraph. 
        We now show that, in this case, $G$ must be isomorphic to $F_3$.
       
        Choose a vertex $x$ with degree $6$. If $G\ncong F_3$, then there exists a neighbor $y$ of $x$ such that $d_y\geq3$. Let $w$ be the unique common neighbor of $x$ and $y$. The remaining two triangles are formed by $x u_1 u_2$ and $x u_3 u_4$. 
        %The local structure is shown as Figure \ref{Delta(G)=6}. 
       % \begin{figure}[h]
        %\centering
%\includegraphics[width=0.52\textwidth,height=0.24\textwidth]{maxsix.png}
     %   \caption{Local structure of $G$ when $\Delta(G)=6$.}
      %  \label{Delta(G)=6}
    %\end{figure}
Consider a  function $f_{11}:N(x)\cup N(y)\rightarrow \mathbb{R}$ given by
    \begin{align*}
        f_{11}(z)= 
        \begin{cases}
           -1 & \text { if } z=u_i ; \\
            0 & \text { if } z=x; \\
            1 & \text { if } z \in N[y] \setminus \{x\} .
           \end{cases}
       \end{align*}
It is clear that $f_{11}\in \Lip(1)$ since  $G$ is $C_4$-free, and $\nabla_{yx} f_{11}=1$.
By \cref{thm::2.1}, we have
\begin{align*}
     0< \kappa_{\mathrm{LLY}}(x,y)
    & \leq \Delta f_{11}(x)-\Delta f_{11}(y)\\
    &\leq \frac{1}{6}(4 \cdot (-1)+1+1-0) - \frac{1}{d_y}(d_y-1-d_y)\\
    %&\leq \frac{4}{d_x}+\frac{1}{d_y}-1\\
    &=-\frac{1}{3} + \frac{1}{d_y} \leq 0,
\end{align*}
which leads to a contradiction. Thus, $G\cong F_3$.
    \end{case}
    It can be verified, either by using the graph curvature calculator or by direct calculation, that the graphs $T,H,F_2,F_3$ are positively LLY-curved. This completes the proof.
\end{proof}

The following corollary is a direct consequence of \cref{thm::1.5}. 

\begin{corollary}
    Let $G$ be a simple connected positively LLY-curved $C_4$-free outerplanar graph with $\delta(G)\geq 2$. 
    Then $G\cong C_3$, $G\cong C_5$ or $G\cong F_k$ $(k=2,3)$.
\end{corollary}

\section{An example}\label{section4}

In our proof, we only calculate the upper bounds of LLY curvature, and utilize the graph curvature calculator to verify whether a graph is positively LLY-curved or not for convenience. 
As an example, here we verify that $F_2$ has positive LLY curvature directly by establishing lower bounds via \eqref{eq:bourne}.

Denote the vertex of degree 4 in $F_2$ by $x$ and two triangles by $x x_1 x_2$ and $x x_3 x_4$. By symmetry, we only need to consider $\kappa_{\mathrm{LLY}}(x,x_1)$ and $\kappa_{\mathrm{LLY}}(x_1,x_2)$. 

For $\frac{1}{3}\leq \alpha\leq 1$, consider a coupling $A_1:V\times V\rightarrow [0,1]$ defined as
\begin{align*}
    A_1(u,v)= 
    \begin{cases}
         \min\{m_x^\alpha(u),m_{x_1}^\alpha(u) \}, & \text { if } u=v ; \\
         \alpha -\frac{1-\alpha}{2}, & \text { if } u=x,v=x_1 ; \\
         \frac{1-\alpha}{2}-\frac{1-\alpha}{4}, & \text { if } u=x_3,v= x_1 ; \\
         \frac{1-\alpha}{2}-\frac{1-\alpha}{4}, & \text { if } u=x_4,v= x_2 ; \\
          0, & \text { otherwise.}
    \end{cases}
\end{align*}
By \cref{eq:bourne}, we derive
\begin{align*}
    \kappa_{\mathrm{LLY}}(x, x_1)=\frac{\kappa_\alpha (x,x_1)}{1-\alpha} \geq \frac{1 - \sum \limits_{v \in V} \sum \limits_{u\in V} A_1(u,v)d(u,v)}{1-\alpha}=\frac{1}{2}.
\end{align*}
Similarly, for $\frac{1}{3}\leq \alpha\leq 1$, consider another coupling $A_2:V\times V\rightarrow [0,1]$ defined as 
\begin{align*}
    A_2(u,v)= 
    \begin{cases}
        \min\{m_{x_1}^\alpha(u),m_{x_2}^\alpha(u) \}, & \text { if } u=v ; \\
        \alpha -\frac{1-\alpha}{2}, & \text { if } u=x_1,v=x_2 ; \\
        0, & \text { otherwise.}
    \end{cases}
\end{align*}
By \cref{eq:bourne}, we deduce
\begin{align*}
    \kappa_{\mathrm{LLY}}(x_1,x_2)=\frac{\kappa_\alpha(x_1,x_2)}{1-\alpha}\geq\frac{1-\sum\limits_{v\in V}\sum\limits_{u\in V} A_2(u,v)d(u,v)}{1-\alpha}=\frac{3}{2}.
\end{align*}

Therefore, $F_2$ is positively LLY-curved.

\section{Concluding remarks}
In fact, the condition $\delta(G)\geq 2$ can be removed. 
\cref{lemma::3.1} implies that any pendant edge has positive LLY curvature, and removing such edges increases LLY curvature on each of their adjacent edges. 
Therefore, for a graph containing pendant edges, we may remove them iteratively, eventually obtaining a graph with no pendant edges. 
To classify all simple connected positively LLY-curved $C_4$-free graphs, it suffices to begin with the graphs listed in \cref{thm::1.5} and attach pendant edges layer by layer, verifying the curvature at each step using the graph curvature calculator (Note that star graphs are all positively-curved, which may be viewed as constructed by attaching pendant edges to a single vertex).

The condition $C_4$-free cannot be weakened to induced $C_4$-free, as many positively LLY-curved graphs contain no induced $C_4$. 
For instance, all the complete graphs have no induced $C_4$ and yet possess large positive LLY curvature.
Other examples, such as $K_2$ joined with a independent set, have been shown in~\cite[Figure 2]{planar LLY}.

At last, combining our main result with the classification of graphs with girth at least 5 given by Lin, Lu and Yau~\cite{girth5}, we aim to pursue a stronger classification result through the following problem.

\begin{problem}
     Classify all  connected $C_4$-free graphs with non-negative LLY curvature.
\end{problem}

%\section*{Acknowledgements}
	%\vspace*{4mm}
%\section*{Data availability}
	%Data sharing is not applicable to this article as no datasets were generated or analyzed during the current study.
%\bibliographystyle{plain}
%\bibliography{reference.bib}

\end{document}